\documentclass[reqno]{amsart}
\usepackage{amssymb, stmaryrd, eucal}
\usepackage{amscd}

\newtheorem{theorem}{Theorem}%[section]
\newtheorem{lemma}[theorem]{Lemma}

\newtheorem{remark}[theorem]{Remark}
\newtheorem{proposition}[theorem]{Proposition}
\newtheorem{corollary}[theorem]{Corollary}

\usepackage{color}  %% USE OF COLORS
\definecolor{darkgreen}{rgb}{0.03, 0.5, 0.03}
\begin{document}

\title[]{A ``$v$-operation free'' approach to  Pr\"ufer $v$-multiplication domains}
\author[ ]{Marco Fontana and Muhammad Zafrullah}

\date{\today}
\address{M.F.: \ Dipartimento di Matematica, Universit\`a degli Studi
``Roma Tre'', 00146 Rome, Italy.}
\email{fontana@mat.uniroma3.it }
\address{M.Z.: \ 57 Colgate Street,
Pocatello, ID 83201-34, Idaho, USA.}
\email{zafrullah@lohar.com}

\begin{abstract} The so called Pr\"ufer $v$-multiplication domains (P$v$MD's) are usually defined as domains whose finitely generated nonzero ideals are $t$-invertible. These domains generalize Pr\"ufer domains and Krull domains. The P$v$MD's are relatively obscure compared to their very well known special cases. One of the reasons could be that the study of P$v$MD's uses the jargon of star operations, such as the $v$-operation and the $t$-operation. In this paper, we provide characterizations of and basic results on P$v$MD's and related notions without star operations.
 \end{abstract}

\maketitle

  %% SECTION
%%%%%%%%%%%%%%%%%%%%%%%

\section{Introduction and Preliminaries} \label{prel}

Pr\"{u}fer $v$-multiplication domains, explicitly introduced in \cite[Griffin (1967)]{Gr} under the name of $v$-multiplication rings, have been
studied a great deal as a generalization   of Pr\"{u}fer domains and Krull domains.
 One of the { attractions} of  Pr\"{u}fer $v$-multiplication domains  is that they share many properties with Pr\"ufer domains and, furthermore,  they are stable in passing to polynomials, unlike Pr\"ufer domains (since a polynomial ring $D[X]$ is a Pr\"ufer domain only in the trivial case, i.e., when $D$ is a field).
 On the other hand, Pr\"{u}fer $v$-multiplication domains
are a special case of $v$-domains, a class of integrally closed domains which has recently attracted new
attention \cite[Anderson-Anderson-Fontana-Zafrullah (2008)]{AAFZ}, \cite[Halter-Koch (2009)]{H-K-09} and \cite[Fontana-Zafrullah (2009)]{FZ}.  The paper \cite[Dieudonn\'{e} (1941)]
{Di-41} provides a clue to where $v$-domains arose as a separate class of rings, though they were not called $v$-domains there.

The notions of $v$-domain and of several of its specializations  may { be obscured by the} jargon of Krull's star operations used in the ``official'' definitions and standard characterizations (the best source available for star operations and for this type of approach to $v$-domains is
Sections 32 and 34 of \cite[Gilmer (1972)]{Gi}). 
 The overhanging presence of  star operations 
%could prevent the use of 
%does not 
seems to have limited the popularization of these distinguished classes of integral domains and, perhaps, has prevented the use of other powerful techniques,
such as those of homological algebra, in their study.

The aim of this note is to
provide ``star operation free'' definitions and characte\-ri\-za\-tions of  the above mentioned classes of
integral domains. In particular, we prove statements that, when used as definitions, do not
mention any star operations, leading to new characterizations of various special classes of $v$-domains.

\medskip

%Note that
%\textquotedblleft officially\textquotedblright\ an integral domain is a $v$%
%-domain (resp., a Pr\"{u}fer $v$-multiplication domain (PVMD)) if for every
%nonzero finitely generated (fractional) ideal $A$ of $D$ we have $%
%(AA^{-1})_{v}=D$ (resp., $(AA^{-1})_{t}=D).$ So to understand the
%justification of the new definitions the reader will need some basic
%knowledge of the star operations. For this the best source available is
%sections 32 and 34 of Gilmer \cite{Gi}. For a quick review the reader may
%look up HD0311 for star operations and note the following.

 Let $D$ be an integral domain with quotient field $K$. Let $
\boldsymbol{\overline{F}}(D)$ be the set of all nonzero
$D$--submodules of $K$ and let $\boldsymbol{F}(D)$ be the set of
all nonzero fractional ideals of $D$, i.e., $A \in
\boldsymbol{F}(D)$ if $A \in \boldsymbol{ \overline{F}}(D)$ and
there exists an element $0 \ne d \in D$ with $dA \subseteq D$. Let
$\boldsymbol{f}(D)$ be the set of all nonzero finitely generated
$D$--submodules of $K$. Then, obviously $\boldsymbol{f}(D)
\subseteq \boldsymbol{F}(D) \subseteq
\boldsymbol{\overline{F}}(D)$.

For $D$-submodules $A,B \in \boldsymbol{\overline{F}}(D)$, we use the notation $(A:B)$ to denote the set $
\{x\in K \mid xB\subseteq A\}$. If $(A:B) \neq (0)$, clearly, $(A:B) \in \boldsymbol{\overline{F}}(D)$ and if  $A \in
\boldsymbol{F}(D)$, then $(A:B) \in \boldsymbol{F}(D)$.  Denote $(D:A)$ by $A^{-1}$, which belongs to $\boldsymbol{F}(D))$
 whenever $A$ does, and $(D:A) =(0)$ if $A \in \boldsymbol{\overline{F}}(D) \setminus \boldsymbol{F}(D)$.  If $A\subseteq B$ then $A^{-1}\supseteq B^{-1}$.  Moreover, from the definition, it follows that $AA^{-1}\subseteq D$ and $D^{-1}=D$. Recall that, for $A\in \boldsymbol{\overline{F}}(D)$, $
A^{v}:=(A^{-1})^{-1}=(D:(D:A))$ and note that, if $A \in \boldsymbol{\overline{F}}(D) \setminus \boldsymbol{F}(D)$, then $A^{v} =K$, since $(D:A) = (0)$. Set $A^{t}:=\bigcup \{F^{v} \mid F\subseteq A \mbox{ and }  F\in \boldsymbol{f}(D)\}$. It can be easily  shown that $(A^{v})^{-1}=A^{-1} =(A^{-1} )^v$.
 If $A\in \boldsymbol{F}(D))$ is such that $A=A^{v}$ (resp., $A=A^{t})$ we say that $A$ is a \emph{fractional $v$-ideal} (resp., a \emph{fractional $t$-ideal}) \emph{of $D$}. 
 Note that,   if $A \in \boldsymbol{\overline{F}}(D) \setminus \boldsymbol{F}(D)$, then   $A=A^{v}$  if and only if $A=K$; 
 %however,
  { on the other hand,  it is possible that $A =A^t\subsetneq K$ for $A \in \boldsymbol{\overline{F}}(D) \setminus \boldsymbol{F}(D)$  (for instance, if $D$ is a \texttt{fgv}-domain, i.e., an integral domain such that every nonzero finitely generated ideal is a $v$-ideal \cite[Zafrullah (1985)]{zafrullah}, then  $A =A^t$ for every $A \in \boldsymbol{\overline{F}}(D)$)}. 
%  (e.g., take a valuation overring $V$ of $D$ such that $(D:V) =(0)$, clearly, in this case, $V^v = K$ and for each $F \in \boldsymbol{f}(D)$ with $F \subseteq V$  $F \subseteq FV = xV$ for some $x \in V$, hence $F^v \subseteq$   \ec  
 
 A fractional $v$-ideal is also called a \emph{fractional divisorial ideal}. If $ A \in \boldsymbol{F}(D)$,
 $A^{-1}$ is a fractional $v$-ideal, and every \emph{fractional invertible ideal}   (i.e., every fractional ideal $A$ such that $AA^{-1} =D$) is both a fractional $v$-ideal and a fractional $t$-ideal. 
 If there is a finitely generated fractional ideal $F$ such that 
$A^v=F^{v}$, we say that $A^v$ is a \emph{fractional $v$-ideal of finite type}. Note that, in this definition, we do not require that $F \subseteq A$; if  there is a finitely generated fractional ideal $F$ such that 
$A^v=F^{v}$ and $F \subseteq A$, we say that $A^v$ is a \emph{fractional $v$-ideal of strict finite type}. Examples of  $v$-ideals of finite type that are not $v$-ideals of strict finite type are given in \cite[Gabelli-Houston (1997), Section (4c)]{GH}.
 If $\ast$ provides here a general notation for the $v$- and 
$t$-operation,
% or the identity $d$-operation on $D$ (where $A^d := A$ for each $ A \in \boldsymbol{F}(D)), 
then  call $A\in \boldsymbol{F}(D)$ \emph{$\ast $-invertible} if there is $B\in \boldsymbol{F}(D)$
such that $(AB)^{\ast }=D$. 
It can be shown that in this case $B^{\ast
}=A^{-1}$. It is obvious that an invertible ideal is  
$t$-invertible and a $t$-invertible ideal is also $v$-invertible.  So, $D$ is called  a \emph{$v$-domain} (resp., a \emph{Pr\"ufer $v$-multiplication domain} (for short, \emph{P$v$MD})) if every $F\in \boldsymbol{f}(D)$ is 
$v$-invertible (resp., $t$-invertible). Both these notions generalize the concept of Pr\"ufer domain, since a Pr\"ufer domain can be characterized by the fact that every $F\in \boldsymbol{f}(D)$ is invertible,   and, at the same time, the concept  of Krull domain  because, as we mention later, a domain D is a Krull domain if and only if every nonzero ideal of D is $t$-invertible.

  It can be shown
that $F\in \boldsymbol{f}(D)$ is $t$-invertible if and only if $F$ is $v$-invertible and $F^{-1}$ is a $v$-ideal of finite type { \cite[Zafrullah (2000), Theorem 1.1(c)]{Zaf-00}}. In particular, 
  from the previous considerations, we deduce:
$$
\mbox{ Pr\"ufer domain \; $\Rightarrow$ \; P$v$MD  \; $\Rightarrow$ \;  $v$-domain.}
$$
It is well known that the converse of  each of the previous implications does not hold in general. For instance, a Krull domain which is not Dedekind (e.g., the polynomial ring ${\mathbb Z}[X]$) shows the irreversibility of the first implication. An example of a $v$-domain which is not a P$v$MD was given in \cite[Dieudonn\'{e} (1941)]{Di-41}.

  %% SECTION
%%%%%%%%%%%%%%%%%%%%%%%

\section{Results} \label{res}

The following result maybe in the folklore. We have taken it from \cite[HelpDesk 0802]{Za-HD}, where the  second named author of the present paper made a limited attempt to define P$v$Md's without the $v$-operation.

%%%%%% LEMMA {v-invertible} 
\begin{lemma}\label{v-invertible} Given an integral domain $D$, a fractional ideal $A\in \boldsymbol{F}(D)$ is $v$-invertible if and only if $
(A^{-1}:A^{-1})=D$.
\end{lemma}

\begin{proof} Suppose that $(A^{-1}:A^{-1})=D$. Let $x\in (AA^{-1})^{-1}\supseteq D$.
Then, $x(AA^{-1})\subseteq D$ or $xA^{-1}\subseteq A^{-1}$ or $x\in
(A^{-1}:A^{-1}) =D$. So, $(AA^{-1})^{-1}\subseteq D$ and we have $(AA^{-1})^{-1}=D.$ This gives $(AA^{-1})^{v}=D$.

Conversely, if $A$ is $v$-invertible, then $(AA^{-1})^{-1}=D$. Let $x\in
(A^{-1}:A^{-1})\supseteq D$. Then, $xA^{-1}\subseteq A^{-1}$. Multiplying both
sides by $A$ and applying the $v$-operation, we get $x\in D$. So, $D\subseteq
(A^{-1}:A^{-1})\subseteq D$ and the equality follows.
\end{proof}

%%%%%%% Theorem {v-domain} 
\begin{theorem}\label{v-domain} The following are equivalent for an integral domain $D$.

\begin{enumerate}

\item [(i)] $D$ is a $v$-domain.

\item [(ii)]  $(F^{-1}:F^{-1})=D$ for each $F\in \boldsymbol{f}(D)$.

\item [(iii)]  $(F^{v}:F^{v})=D$ for each $F\in \boldsymbol{f}(D)$.

\item [(iv)]  $((a,b)^{-1}:(a,b)^{-1})=D$ for each two generated fractional ideal $(a,b)\in \boldsymbol{f}(D)$.

\item [(v)]  $((a)\cap (b)):((a)\cap (b))=D$ for all $a,b\in D\backslash \{0\}$.
\end{enumerate}
\end{theorem}

\begin{proof} (i)$\Leftrightarrow$(ii) follows from Lemma \ref{v-invertible} and from the definition of { a} $v$-domain.

(i)$\Rightarrow$(iii). Let $F\in \boldsymbol{f}(D)$ and  $x\in (F^{v}:F^{v})\supseteq D$. Then, $xF^{v}\subseteq F^{v}$.
 Multiplying both sides by $F^{-1}$ and applying the $v$-operation, we get $x(F^{v}F^{-1})^{v}\subseteq (F^{v}F^{-1})^{v}$. But, by
(i), $(F^{v}F^{-1})^{v}=(FF^{-1})^{v}=D$ and so $x\in D$. This forces $D\subseteq
(F^{v}:F^{v})\subseteq D$.

(iii)$\Rightarrow $(i). Let $F\in \boldsymbol{f}(D)$ and $x\in (F^{v}F^{-1})^{-1}\supseteq D$. Then, $x(F^{v}F^{-1})\subseteq D$. But then $xF^{v}\subseteq F^{v}$, which gives $x\in (F^{v}:F^{v})=D$. Therefore $D\subseteq (F^{v}F^{-1})^{-1}\subseteq D$,
which means that every $F\in \boldsymbol{f}(D)$ is $v$-invertible.

(ii)$\Rightarrow$(iv) is obvious.

(iv)$\Rightarrow$(v). Let $a,b\in D$ be two nonzero elements and, by (iv), let $((a,b)^{-1}:(a,b)^{-1})=D$.  Since $(a,b)^{-1} = (D:(a,b)) = (D:(a)) \cap (D:(b)) = (a^{-1}) \cap (b^{-1}) = a^{-1}b^{-1}((a) \cap (b))$, then from the assumption  we have
 $(a^{-1}b^{-1}((a)\cap (b)):a^{-1}b^{-1}((a)\cap (b)))= D$ which is the same as $
((a)\cap (b)):((a)\cap (b))=D$, for all $a,b\in D\backslash \{0\}$.

(v)$\Rightarrow$(i). Recall that $D$ is a $v$-domain if and only if every
two generated nonzero ideal of $D$ is $v$-invertible \cite[Mott-Nashier-Zafrullah (1990), Lemma 2.6]{MNZ}.  (Note that H. Pr\"ufer proved that every $F\in \boldsymbol{f}(D)$ is invertible if and only if every
two generated nonzero ideal of $D$ is invertible \cite[Pr\"ufer (1932),   page 7]{Prufer}; a similar result, for the $t$-invertibility case, was proved in \cite[Mott-Nashier-Zafrullah (1990), Lemma 1.7]{MNZ}.)
Now, let $a,b\in D\backslash \{0\}$ and $
x\in ((a,b)(a,b)^{-1})^{-1} \supseteq D$. Then $x(a,b)(a,b)^{-1}\subseteq D$, or $x(a,b)^{-1}\subseteq (a,b)^{-1}$, or 
$xa^{-1}b^{-1}{((a)\cap (b))}\subseteq a^{-1}b^{-1}{((a)\cap (b))}$.  This is equivalent to $x((a)\cap
(b))\subseteq (a)\cap (b)$ or $x\in (((a)\cap (b)):((a)\cap (b)))=D$. This
forces $D\subseteq ((a,b)(a,b)^{-1})^{-1}\subseteq D$.
\end{proof}

Call an integral domain $D$ a \emph{$v$-finite conductor} (for short, a \emph{$v$-{\texttt{FC}}-}) \emph{domain} if $
(a)\cap (b)$ is a $v$-ideal of finite type, for every pair $a,b\in
D\backslash \{0\}$ { \cite[Dumitrescu-Zafrullah (2008), Section 2]{DZ}}.

The above definition of $v$-{\texttt{FC}}-domain makes use of the $v$-operation. We have a somewhat contrived solution for
this, in the form of the following characterization of $v$-{\texttt{FC}}-domains.

%%%%%%% PROPOSITION {v-FC-domain}
\begin{proposition} \label{v-FC-domain}
An integral domain $D$ with quotient field $K$ is a $v$-{\texttt{FC}}-domain if and only if for
each pair $a,b$ in $D\backslash \{0\}$ there exist $y_{1},y_{2},...,y_{n}\in
K\backslash \{0\}$, with $n \geq 1$, such that $(a,b)^{v}=\bigcap \{ y_{i}D \mid 1 \leq i \leq n\}$. Consequently, $D$
is a $v$-{\texttt{FC}}-domain if and only if for each pair $a,b$ in $D\backslash \{0\}$
there exist $z_{1},z_{2},...,z_{m}\in K\backslash \{0\}$, with $m\geq 1$,  such that $((a)\cap
(b))^{-1}=\bigcap \{z_{j}D \mid  1 \leq j \leq m\}$.
\end{proposition}

\begin{proof}
 Let $D$ be a $v$-{\texttt{FC}}-domain and let $a,b\in D\backslash \{0\}.$ Then,
there are $a_{1}, a_{2}, ..., a_{n}\in D$ such that $(a)\cap
(b)=(a_{1}, a_{2}, ...,a_{n})^{v}$. Dividing both sides by $ab$,  we get $ (a,b)^{-1} =a^{-1}b^{-1} {(
(a)\cap (b))}=(a_{1}/ab, a_{2}/ab, ..., a_{n}/ab)^{v}$. This gives 
$$
(a,b)^{v}=((a_{1}/ab, a_{2}/ab, ..., a_{n}/ab)^{v})^{-1}=
% (a_{1}/ab, a_{2}/ab, ..., a_{n}/ab)^{-1}= 
\bigcap \left\{ \frac{ab}{a_{i}}D \mid 1\leq i \leq n \right\}.$$

 Conversely,  if for
each pair $a,b$ in $D\backslash \{0\}$ there exist $y_{1},y_{2},...,y_{n}\in
K\backslash \{0\}$ such that $(a,b)^{v}=\bigcap \{y_{i}D \mid 1\leq i \leq n \}$, then 
$$a^{-1}b^{-1}{((a)\cap (b))} = (a,b)^{-1}=((a,b)^{v})^{-1}= \left(\bigcap\{y_{i}D \mid 1\leq i \leq n\}\right)^{-1}.$$

On the other hand,  $(\bigcap\{y_{i}D \mid 1\leq i \leq n\})^{-1} = (y_1^{-1}D, y_2^{-1}D, ..., y_n^{-1}D)^v$ \cite[Houston-Zafrullah (1988), Lemma 1.1]{HZ}, and 
 this gives 
 %$\frac{(a)\cap (b)}{ab}%
%=(\dsum \frac{1}{y_{i}}D)_{v}$ by \cite[Lemma 1.1]{HZ1}. Or
$(a)\cap
(b) = \left(\frac{ab}{y_{1}}D,  \frac{ab}{y_{2}}D, ...,  \frac{ab}{y_{n}}D\right)^{v}$. 
%since $(D: (\frac{ab}{y_{1}}D,  \frac{ab}{y_{2}}D, ...,  \frac{ab}{y_{n}}D) = \bigcap \{ (D: %\frac{ab}{y_{i}}D \mid 1\leq i \leq n \}$ and 

 For the \textquotedblleft
consequently\textquotedblright\ part,  note that $((a)\cap (b))^{-1}=a^{-1}b^{-1}(a,b)^{v}.$
\end{proof}

An immediate consequence of the above results is the following
characterization of P$v$MD's, in which statements (iii) and (iv) are ``$v$-operation free''.  
 %%%% COROLLARY {PvMD}
\begin{corollary} \label{PvMD}  The following are equivalent for an integral domain $D$.

\begin{enumerate}
\item [(i)] $D$ is a P$v$MD.
\item [(ii)]  $D$ is a $v$-domain and a $v$-{\texttt{FC}}-domain,

\item [(iii)] for all $a,b\in D\backslash (0)$,  $((a)\cap (b))^{-1}$ is a finite
intersection of principal fractional ideals and $((a)\cap (b)):((a)\cap
(b))=D$.

\item [(iv)] for all $a,b\in D\backslash (0)$, $((a)\cap (b))^{-1}$ is  a finite
intersection of principal fractional ideals and $((a,b)^{-1}:(a,b)^{-1})=D$.

%\item [(iii)] $D$ is a $v$-{\texttt{FC}}-domain and for all $a,b\in D\backslash (0)$ we have $%
% ((a)\cap (b)):((a)\cap (b))=D.$

\end{enumerate}
\end{corollary}
\begin{proof} (i)$\Leftrightarrow$(ii)
 stems from the fact that $D$ is a P$v$MD (resp., a $v$-domain) if
and only if every two generated nonzero ideal of $D$ is $t$-invertible (resp., $v$-invertible) \cite[Malik-Mott-Zafrullah (1988), Lemma 1.7]{MMZ}  (resp.,   \cite[Mott-Nashier-Zafrullah (1990), Lemma 2.6]{MNZ}). Moreover,  every two generated ideal  of $D$ is 
$t$-invertible if and only if 
every two generated ideal $(a,b)$ of $D$ is 
$v$-invertible and such that $(a, b)^{-1}= (x_{1},x_{2},...,x_{r})^{v}$ where $r\geq 1$   and $ x_{1}, x_{2},..., x_{r} \in K$ { \cite[Zafrullah (2000), Theorem 1.1(c)]{Zaf-00}}. 
Finally, since  $(a,b)^{-1}= a^{-1}b^{-1}{((a)\cap (b))}$,  $(a, b)^{-1}$ is a  fractional $v$-ideal of finite type 
%f and only if $(a, b)^{-1}= (x_{1},x_{2},...,x_{r})^{v}$
 if and only if  $(a)\cap (b)$ is a $v$-ideal of finite type.
%=(x_{1},x_{2},...,x_{r})^{v}$ and, multiplying both sides by $ab$, we have 
%$(a)\cap (b)=(abx_{1},abx_{2},...,abx_{r})^{v}$.  On the other hand, if $(a)\cap (b)=(y_%{1}, y_{2}, ...,y_{s)^{v}$ for $ y_{1}, y_{2},..., y_{s} \in (a)\cap (b)$, then $(a,b)^{-1} = 
%((ab)^{-1}y_{1}, (ab)^{-1}y_{2}, ...,(ab)^{-1}y_{s})^{v}$.
%%Recall from the introduction that a
%fractional $v$-ideal $A$ is said to be a $v$-ideal of finite type if there
%exist $a_{1},a_{2}...a_{n}\in A$ such that $A=(a_{1},a_{2}...a_{n})_{v}.$

(ii)$\Leftrightarrow$(iii) and (ii)$\Leftrightarrow$(iv) are straightforward consequences of 
Theorem \ref{v-domain} and Proposition \ref{v-FC-domain}.
\end{proof}

Recall that an integral domain $D$ is called a \emph{finite conductor} (for short, \emph{{\texttt{FC}}-}) \emph{domain} if $((a)\cap (b))$
is finitely generated for each pair $a, b\in D$.  Just to show how far we have
traveled since 1978, when this notion was introduced, we state and provide an easy proof the following statement, which appeared as the main result (Theorem 2) in  \cite[Zafrullah (1978)]{Za-78}.

%%%COROLLARY {FC-PvMD}
\begin{corollary} \label{FC-PvMD}  An integrally closed {\texttt{FC}}-domain is a P$v$MD.
\end{corollary}

\begin{proof} First note that, since $D$ is integrally closed $(F:F)=D$ for every
finitely generated ideal $F$ of $D$ \cite[Gilmer (1972), Theorem 34.7]{Gi}. So, for each pair $a,b\in D\backslash
\{0\},$ since $D$ is a ${\texttt{FC}}$-domain, $((a)\cap (b)):((a)\cap (b))=D$. But this makes $
D $ a $v$-domain  by Theorem \ref{v-domain} and, so, a P$v$MD by Corollary \ref{PvMD}.
\end{proof}

%While the result in Corollary \ref{PvMD} was the main result of \cite[...][Zafrullah (1978)]{Za-78}
%that paper contained many other useful techniques. For example, \cite[....]{Za-78}
%was the first in the literature to introduce the formula 
%$
%(FD_{S})^{v}=(F^{v}D_{S})^{v}$ where $F$ is a finitely generated (nonzero)
%ideal of $D$ and $S$ is a multiplicative set of $D,$ among other useful
%techniques.

Lemma \ref{v-invertible} can also be instrumental in characterizing completely integrally
closed (for short, CIC-) domains (see, for instance,
\cite[Gilmer (1972), Theorem 34.3]{Gi}). Also the previous approach leads to a
characterization of Krull domains in a manner similar to the
characterization of $v$-domains leading to the characterization of P$v$MD's.

%%%% PROPOSITION {CIC-domain}
\begin{proposition} \label{CIC-domain} The following are equivalent for an integral domain $D$.
\begin{enumerate}
\item[(i)] $D$ is a CIC-domain.

\item[(ii)] $(A^{-1}:A^{-1})=D$ for all $A\in\boldsymbol{F}(D)$.
\end{enumerate}
In particular, a CIC-domain is a $v$-domain.
\end{proposition}
\begin{proof} Note that $D$ is CIC if and only if every $A\in \boldsymbol{F}(D)$ is $v$-invertible \cite[Gilmer (1972), Proposition 34.2 and Theorem 34.3]{Gi}. Now, the equivalence (i)$\Leftrightarrow$(ii) is an immediate consequence of Lemma \ref{v-invertible}.
The last statement is a straightforward consequence of the equivalence (i)$\Leftrightarrow$(ii) of Theorem \ref{v-domain}.
\end{proof}

%%%%%  REMARK

\begin{remark} \rm  We have been informed by the referee that he/she has used Proposition \ref{CIC-domain} while teaching a course on multiplicative ideal theory. So, like Lemma \ref{v-invertible}, this is another folklore result in need of a standard reference.\end{remark}

%%%%% PROPOSITION {Krull}
\begin{theorem} \label{Krull}
The following are equivalent for an integral domain $D$.
\begin{enumerate}
\item[(i)] $D$ is a Krull domain.

\item[(ii)] $D$ is a Mori  $v$-domain.

\item[(iii)]  For each $A\in \boldsymbol{F}(D)$, there exist $y_{1},y_{2}, ...,y_{n}\in A$ such that $A^{-1}=\bigcap \{{y_{i}}^{-1}D \mid 1\leq i \leq n \}$ and, for all $a,b\in D\backslash \{0\}$, $
((a)\cap (b)):((a)\cap (b))=D$.

\item[(iv)]  For each $A\in \boldsymbol{F}(D)$, there exist $x,y\in A$ such that $A^{-1}=x^{-1}
D\cap {y}^{-1}D$ and for all $a,b\in D\backslash \{0\},$ $((a)\cap
(b)):((a)\cap (b))=D$.
\end{enumerate}
\end{theorem}

Before we prove Theorem \ref{Krull}, it seems pertinent to give some introduction.
For a quick review
of Krull domains,  the reader may consult the first few pages of \cite[Fossum (1973)]{Fo}. A
number of characterizations of Krull domains can be also found in  \cite[Houston-Zafrullah (1988), Theorem
2.3]{HZ}. The one that we can use here is: \emph{$D$ is a Krull domain if and
only if each $A\in \boldsymbol{F}(D)$ is $t$-invertible}. Which means, as observed above, that \emph{$D$ is a Krull
domain if and only if for each $A\in \boldsymbol{F}(D),$ $A$ is $v$-invertible and $
A^{-1}$ is a fractional $v$-ideal of finite type}. In particular, we reobtain that a Krull domain is a  a P$v$MD (and so, in particular, a $v$-domain).

An integral domain $D$ is called a \emph{Mori domain} if $D$ satisfies the ascending chain condition  on
integral divisorial ideals { (see, for instance, \cite[Barucci (2000)]{Ba})}. Different aspects of Mori domains were stu\-died
by Toshio Nishimura in a series of papers. For instance, in  \cite[Nishimura (1967), Theorem, page 2]{Ni2},  he
showed that \emph{a domain $D$ is a Krull domain if and only if $D$ is a Mori
domain and completely integrally closed}.
  For another proof of this result, see \cite[Zafrullah (1989), Corollary 2.2]{Zaf}.

 On the other hand, an integral domain $D$ is a Mori
domain if and only if, for each $A\in \boldsymbol{F}(D)$, $A^v$ is a fractional $v$-ideal of strict finite type \cite[Nishimura (1962), Lemma 1]{Ni1}. 
A variation of this characterization is given next.

%%%%%% LEMMA {Mori}
\begin{lemma} \label{Mori}  Let $D$ be an integral domain. Then,  $D$ is Mori if and only if 
 for each $
A\in \boldsymbol{F}(D)$ there exist $y_{1},y_{2},...,y_{n}\in A\backslash \{0\}$, with $n \geq 1$,  such
that $A^{-1}=\bigcap \{{y_{i}}^{-1}D \mid 1\leq i \leq n\}$.
\end{lemma}
\begin{proof}  As we observed above,   $D$ is a Mori
domain if and only if for each $A\in \boldsymbol{F}(D)$ there exist $
y_{1},y_{2},...,y_{n}\in A\backslash \{0\}$ such that $
A^{v}=(y_{1},y_{2},...,y_{n})^{v}$. This last equality is equivalent to
$
A^{-1}=(y_{1},y_{2},...,y_{n})^{-1} = \bigcap \{{y_{i}}^{-1}D \mid 1\leq i \leq n\},$
since, by \cite[Houston-Zafrullah (1988), Lemma 1.1]{HZ}, we have 
$$
\left(\bigcap \{y_{i}^{-1}D \mid 1\leq i \leq n\}\right)^{-1} =   \left((y_{1}^{-1})^{-1},(y_{2}^{-1})^{-1},...,(y_{n}^{-1})^{-1}\right)^{v} = (y_{1},y_{2},...,y_{n})^{v}\,.
$$
\end{proof}

\emph{Proof of Theorem \ref{Krull}}. (i)$\Rightarrow$(ii)  because we already observed that a Krull domain is a CIC Mori domain. Moreover, a CIC-domain is a $v$-domain (Proposition \ref{CIC-domain}).

(ii)$\Rightarrow$(i)  We want to prove that, for each $A\in \boldsymbol{F}(D),$ $A$ is $v$-invertible and $
A^{-1}$ is a fractional $v$-ideal of finite type. The second property is a particular case of the assumption that every fractional divisorial ideal of $D$ is a $v$-ideal of finite type. For the first property, we have that,  for each $A\in \boldsymbol{F}(D)$, there exist $
F\in \boldsymbol{f}(D)$,  with $F \subseteq A$, such that $A^v = F^v$ (or, equivalently, $A^{-1} = F^{-1}$). Since $D$ is a $v$-domain, we have $D = (FF^{-1})^v = (F^vF^{-1})^v = (A^vF^{-1})^v = (AA^{-1})^v$.

 (ii)$\Leftrightarrow$(iii)  is a straightforward consequence of Lemma \ref{Mori} and Theorem \ref{v-domain}((i)$\Leftrightarrow$(v)).
 
  (iii)$\Rightarrow$(iv) follows form the fact that (iii)$\Leftrightarrow$(i) and,  if $D$ is a Krull
domain, then for every $A\in \boldsymbol{F}(D)$ there exist $x,y\in A$ such that $
A^{v}=(x,y)^{v}$ { \cite[Mott-Zafrullah (1991), Proposition 1.3]{MZ}}.  Therefore,  $A^{-1}= (x,y)^{-1} = x^{-1}D \cap y^{-1}D$.

 (iv)$\Rightarrow$(iii)  is trivial.
\hfill $\Box$

%%%%%% REMARK {v-finite}
\begin{remark} \rm {\bf (a)}
 In (iii) of Theorem \ref{Krull},  we cannot say that for every $A\in \boldsymbol{F}(D)$
the inverse $A^{-1}$ is expressible as a finite intersection of principal
fractional ideals, because this would be equivalent to $A^{v}$ being of
finite type for each $A\in \boldsymbol{F}(D)$. But there do exist non-Mori domains $D$
such that $A^{v}$ is of finite type for all $A\in \boldsymbol{F}(D)$. For a discussion of those examples you may consult  \cite[Zafrullah (1986), Section 2]{Zafrullah} and \cite[Gabelli-Houston (1997), Section (4c)]{GH}.

{\bf (b)}  Note that a Mori domain is obviously  a $v$-\texttt{FC}-domain, since in a Mori domain every divisorial ideal is a $v$-ideal of (strict) finite type. Therefore,  the equivalences (i)$\Leftrightarrow$(ii) of Theorem \ref{Krull} and  of Corollary \ref{PvMD} shed new light on the relations between P$v$MD's and Krull domains in the class of $v$-domains.

\end{remark}

While several of the above results provide characterizations of  Pr\"ufer $v$-multiplica\-tion domains, $v$-domains, Mori and Krull { domains}, without using Krull's theory of star operations, they
%also indicate the importance of star operations as a
%means of getting deeper than where the mainstream techniques %could not help.
%Indeed, mere definition may not be useful when it comes to using %the notion.
%We
do not diminish the importance of star operations in any way. After all, it was the star operations that developed the notions mentioned above this far. An interested reader will have to extend this work further so that mainstream techniques could be used. To make a start in that direction, we  
give below some further ``star operation free'' characterizations of P$v$MD's, besides
the ones we have already given above.
\medskip

Given an integral domain  $D$, a \emph{prime ideal} $P$ is called \emph{essential for $D$} if $D_{P}$ is a valuation domain and the \emph{domain}
$D$ is called \emph{essential} if there is a family of essential primes $\{P_{\alpha }\}$
for $D$ such that $D =\bigcap D_{P_{\alpha }}$.  Also, call a prime ideal $P$ of $D$
an \emph{associated prime of a principal ideal} if $P$ is a minimal prime over a
proper nonzero ideal of the type $((a):_{D}(b))$, for some $a, b \in D$. The associated primes of
principal ideals have been discussed in \cite[Brewer-Heinzer (1974)]{BH},
where it was also shown that if $S$ is a multiplicative set in $D$ then $
D_{S}=\bigcap D_{P}$ where $P$ ranges over the associated primes of principal
ideals disjoint from $S$ \cite[Brewer-Heinzer (1974), Proposition 4]{BH}. For brevity,  we call here an associate prime of a
principal ideal of $D$ simply an \emph{associated prime of $D$}. 

Following 
\cite[Mott-Zafrullah (1981)]{MZa}, call $D$ a \emph{\texttt{P}-domain} if every associated prime of $D$ is essential. It is easy to see that a \emph{\texttt{P}-domain is an essential domain}. More precisely, it
was shown in  \cite[Mott-Zafrullah (1981),  Proposition 1.1]{MZa} that $D$ is a  \emph{{\texttt{P}-domain} if and only if $D$ is essential
and every quotient ring of $D$ is essential}.  Also,  if $D$ is a {\texttt{P}-domain}
then so are the rings of fractions of $D$ and the rings of polynomials over $
D$ \cite[Mott-Zafrullah (1981), Corollary 1.2]{MZa}. From \cite[Mott-Zafrullah (1981), Corollary 1.4 and Example 2.1]{MZa} one can also get the information that a P$v$MD is a
{\texttt{P}-domain}, but not conversely. 

We now state a result that is already
known but that can be of use if someone wants to deal with P$v$MD's without
having to use, in statements (ii) and (iii),  the star operations.

%%%%% PROPOSITION  {PvMD-2}
\begin{proposition} \label{PvMD-2}
The following are equivalent for an integral domain $D$.

\begin{enumerate}
\item[(i)] $D$ is a P$v$MD.

\item[(ii)] $D$ is a {\texttt{P}-domain}-domain such that, for every pair $a,b\in D\backslash \{0\}$, $
((a)\cap (b))^{-1}$ is a finite intersection of principal fractional ideals.

\item[(ii$^\prime$)]$D$ is a {\texttt{P}-domain}-domain and a $v$-{\texttt{FC}}-domain.

\item[(iii)] $D$ is an essential domain such that, for every pair $a,b\in D\backslash
\{0\}$, $((a)\cap (b))^{-1}$ is a finite intersection of principal fractional
ideals.
\item[(iii$^\prime$)] $D$ is an essential  $v$-{\texttt{FC}}-domain.
\end{enumerate}
\end{proposition}

\begin{proof} As we already mentioned above, from \cite[Mott-Zafrullah (1981)]{MZa} we  know that a P$v$MD is a
 {\texttt{P}-domain}  and that a  {\texttt{P}-domain}  is essential. 
 Moreover, from Corollary \ref{PvMD},  if $D$ is a P$v$MD, we have, for every pair $a,b\in
D\backslash \{0\}$,  that $((a)\cap (b))^{-1}$ is a finite intersection of
principal fractional ideals  (or, equivalently, $D$ is a $v$-{\texttt{FC}}-domain, by Proposition \ref{v-FC-domain}). Therefore, (i)$\Rightarrow$(ii) $\Rightarrow$(iii),  (ii)$\Leftrightarrow$(ii$^\prime$) and (iii)$\Leftrightarrow$(iii$^\prime$).

 (iii)$\Rightarrow$(ii). Recall that, from   \cite[Kang (1989), Lemma 3.1]{Kang}, we have that an essential domain is a $v$-domain { (the reader may also want to consult the survey paper \cite[Fontana-Zafrullah (2009), Proposition 2.1]{FZ} and, for stricly related results, \cite[Zafrullah (1988), Lemma 4.5]{Zafr} and \cite[Zafrullah (1987), Theorem 3.1 and Corollary 3.2]{Za-87})}.  The conclusion follows from Corollary \ref{PvMD}((ii)$\Rightarrow$(i)) (and Proposition \ref{v-FC-domain}).
\end{proof}

%Proof. By Corollary 4, if $D$ is a PVMD then for every pair $a,b\in
%D\backslash \{0\}$ $((a)\cap (b))^{-1}$ is a finite intersection of
%principal fractional ideals and from \cite{MZa} we know that a PVMD is a
%P-domain and that a P-domain is essential, so (1) $\Rightarrow $ (2) $%
%\Rightarrow $ (3). That an essential domain is a $v$-domain was essentially
%shown in \cite[Lemma 4.5]{Zafr} but the reader may also want to consult \cite%
%{FZ}, and this gives (3) $\Rightarrow $ (4). Finally for (4) $\Rightarrow $
%(1) the reader may look up Corollary 4.

\begin{remark} \rm   Note that, from the proof of Proposition \ref{PvMD-2}((iii)$\Rightarrow$(ii)), we have that each of the statements of Proposition \ref{PvMD-2} is equivalent to
\begin{enumerate}
 \item[(iv)] \emph{$D$ is a $v$-domain such that, for every pair $a,b\in D\backslash \{0\}$,
$((a)\cap (b))^{-1}$ is a finite intersection of principal fractional ideals}.
\end{enumerate}
which is obviously also equivalent to (ii) of Corollary \ref{PvMD}.
\end{remark}

%REMARK  \item[(iv)] $D$ is a $v$-domain such that for every pair $a,b\in D\backslash \{0\}$ 
%$((a)\cap (b))^{-1}$ is a finite intersection of principal fractional ideals.

%The reason why I have brought in associated primes of $D$ is that these
%primes figure prominently in the study of divisorial ideals, though mostly
%in the context of Noetherian rings, in Wolmer Vasconcelos' \cite{V}. Of
%course it is easier to study modules over valuation domains, in light of the
%work of Fuchs and Salce \cite{FS}. It is my hope that a study of modules
%over PVMD's or at least over P-domains, following Vascancelos' approach will
%lead to useful work in non-Noetherian ring theory. 
\medskip 

\noindent {\bf Acknowledgment.} We are thankful to the referee for his/her contribution to the improvement of the presentation of this paper.

%%%%%%%%%%%%%%%%%%%%%%%%%%%%%%%%%%%%%%%%%%
%%%% BIBIOGRAPHY
%%%%%%%%%%%%%%%%%%%%%%%%%%%%%%%%%%%%%%%%%% %%%%%%%%%%%%%%%%%%%%%%%%%%%%%%%%%%%%%%%%%%

\end{document}